\documentclass[a4paper,10pt]{amsart}
\usepackage[utf8]{inputenc}
\usepackage{amsmath,amsthm,mathtools,amssymb, dsfont,bm}
\usepackage{hyperref,xparse}
\usepackage{wasysym,amsopn}
\usepackage{array}
\usepackage{a4wide}
\usepackage{multicol,breqn,enumerate}
\usepackage{pictexwd,dcpic,aliascnt}
\usepackage{tikz}
\usetikzlibrary{arrows,positioning}

\newtheorem{theorem}{Theorem}[section]

\newaliascnt{lemma}{theorem}
\newtheorem{lemma}[lemma]{Lemma}
\aliascntresetthe{lemma}

\newaliascnt{fact}{theorem}
\newtheorem{fact}[fact]{Fact}
\aliascntresetthe{fact}

\newaliascnt{corollary}{theorem}

\aliascntresetthe{corollary}

\newaliascnt{proposition}{theorem}

\aliascntresetthe{proposition}

\newaliascnt{df}{theorem}
\theoremstyle{definition}\newtheorem{df}[df]{Definition}
\aliascntresetthe{df}

\newaliascnt{remark}{theorem}
\theoremstyle{remark}
\aliascntresetthe{remark}

\numberwithin{equation}{section}
\newcommand{\GG}{\mathds{G}}
\newcommand{\HH}{\mathds{H}}
\newcommand{\KK}{\mathds{K}}
\newcommand{\XX}{\mathds{X}}
\newcommand{\YY}{\mathds{Y}}

\newcommand{\CA}{$C^{\ast}$-algebra}

\newcommand{\nc}{{\text{---}\scriptscriptstyle{\|\cdot\|}}}

\DeclareMathOperator{\spn}{span_{\mathds{C}}}
\DeclareMathOperator{\Ob}{Ob}
\DeclareMathOperator{\Mor}{Mor}

\DeclareMathOperator{\Pol}{\mathcal{O}}
\DeclareMathOperator{\id}{id}

\DeclareDocumentCommand\CuS{ m }{C^u(S_{#1}^+)}
\DeclareDocumentCommand\BC{ m }{\mathsf{B}(\mathds{C}^{#1})}
\DeclareDocumentCommand\Cu{ m g }{%
    {\IfNoValueT {#2} {C^u(\mathds{#1})} 
    \IfNoValueF {#2} {C^u(\mathds{#1}_{#2})} }
    }

\title[$I_{k,n}^+$ generate $S_n^+$]{Quantum Increasing Sequences generate Quantum Permutation Groups}
\author{Pawe{\l} J\'oziak}
\address{Faculty of Mathematics and Information Science, Warsaw University of Technology, ul. Koszykowa 75, 00-662 Warszawa, Poland.}
\email{p.joziak@mini.pw.edu.pl}

\begin{document}

\begin{abstract}We answer a question of A.~Skalski and P.M.~So{\l}tan (2016) about inner faithfulness of the S.~Curran's map of extending a quantum increasing sequence to a quantum permutation in full generality. To do so, we exploit some novel techniques introduced by Banica (2018) and Brannan, Chirvasitu, Freslon (2018) concerned with the Banica's conjecture regarding quantum permutation groups. Roughly speaking, we find a inductive setting in which the inner faithfulness of Curran's map can be boiled down to inner faithfulness of similar map for smaller algebras and then rely on inductive generation result for quantum permutation groups of Brannan, Chirvasitu and Freslon.\end{abstract}
\maketitle

\section*{Introduction}
Let $\GG$ be a compact quantum group (in the sense of Woronowicz, but throughout the note we will not need any of the functional analytic features of the associated Hopf-\CA), let $\Pol(\GG)$ be its associated coordinate ring and assume $\beta\colon\Pol(\GG)\to\mathcal{B}$ is a ${}^*$-representation of $\Pol(\GG)$ as a ${}^*$-algebra in some ${}^*$-algebra. Via abstract Gelfand-Naimark duality, such a map corresponds to a map $\hat{\beta}\colon\mathds{X}\to\GG$ and it is natural to ask what is the smallest quantum subgroup containing $\hat{\beta}(\mathds{X})$, or -- in other words -- what the quantum subgroup generated by $\hat{\beta}(\mathds{X})\subset\GG$ is. The answer to this type of questions was studied earlier in \cite{Ban14,BB10,BCV,SS16} in the case of compact quantum groups and later extended to locally compact quantum groups in \cite{PJphd,JKS}.

The concept of a subgroup is central to treating quantum groups from the group-theoretic perspective and many efforts were made to provide accurate descriptions of various aspects of this concept, as well as providing some nontrivial examples, see e.g. \cite{BB09,BY14,DKSS,Pod95}. This note deals with subgroups of the quantum permutation groups, introduced by Wang in \cite{Wang98}. It was observed in \cite{KS09} that quantum permutations can be used to study distributional symmetries of infinite sequences of non-commutative random variables that are identically distributed and free modulo the tail algebra, thus extending the classical de Finetti's theorem to the quantum/free realm.

Another extension of de Finetti's theorem was given by Ryll-Nardzewski: he observed that instead of symmetry of joint distributions under permutations of random variables it is enough to consider subsequences and compare these type of joint distributions to obtain the same conclusion. What this theorem really boils down to is the fact that one can canonically treat the set $I_{k,n}$ of increasing sequences (of indices) as a subset of all permutations $S_n$, and this subset is big enough to generate the whole symmetric group: $\langle I_{k,n}\rangle=S_n$, unless $k=0$ or $k=n$.

This viewpoint was utilised in \cite{Cur11} by Curran to extend theorem of Ryll-Nardzewski to the quantum case: he introduced the quantum space of quantum increasing sequences $I^+_{k,n}$ and described how to canonically extend a quantum increasing sequence to a quantum permutation in $S_n^+$. The analytic properties of the \CA\ $C(I^+_{k,n})$ were strong enough to provide an extension of Ryll-Nardzewski to the quantum/free case. However, these results did not say anything about the subgroup of quantum permutation group that is generated by quantum increasing sequences.

If the analogy with the classical world is complete, one would expect that in fact $\overline{\langle I^+_{k,n}\rangle}=S_n^+$ for all $n$ and $k\neq0,n$. This was ruled out already in \cite{SS16}, where it was observed that $\overline{\langle I^+_{k,n}\rangle}=S_n$ whenever $k=1,n-1$. The second best thing one could hope for is that $\overline{\langle I^+_{k,n}\rangle}=S_n^+$ for at least one $k\in\{2,\ldots,n-2\}$, as this would explain the results of Curran in a more group-theoretic manner. In general, \cite[Question 7.3]{SS16} asks for the complete description of all $\overline{\langle I^+_{k,n}\rangle}$ and emphasises the case $n=4$ and $k=2$ as the first non-trivial case to study. We gave a positive answer in this case in \cite{remark} and explained that in general $S_n\subset\overline{\langle I^+_{k,n}\rangle}\subset S_n^+$, the left inclusion being strict whenever $2\leq k\leq n-2$. Banica's conjecture asks whether there exists a compact quantum group $\mathds{G}$ satisfying $S_n\subset \mathds{G}\subset S_n^+$ such that both inclusions are strict: conjecturally only one of them should be. The conjecture found its positive answer for $n=4$ at \cite{BB09}, and only recently for $n=5$ by providing a deep connection to the subfactor theory, see \cite{Banica18}. The latter was further utilised to argue a inductive-type generation result for quantum permutation groups in \cite{BCF}. 

In this manuscript we use those result to show that $\overline{\langle I^+_{k,n}\rangle}=S_n^+$ for all $n\geq4$ and $2\leq k\leq n-2$, answering a question of Skalski and Sołtan from \cite{SS16}. To do so, we extend the criterion we gave in \cite{remark} to cover a wider class of subsets of the generating sets. Informally, we show that $I_{k,n}\subset I^+_{k,n}$ and that $I^+_{k,n-1}, I^+_{k-1,n-1}\subset I^+_{k,n}$. Consequently, $\overline{\langle I^+_{k-1,n-1}, I_{k,n}\rangle}, \overline{\langle I^+_{k,n-1}, I_{k,n}\rangle} \subset \overline{\langle I^+_{k,n}\rangle}$. This enables us to put ourselves in the inductive generation framework: assuming $\overline{\langle I_{k,n-1}^+\rangle}=\overline{\langle I_{k-1,n-1}^+\rangle}=S_{n-1}^+$ and that $\langle I_{k,n-1}\rangle=S_n$, we deduce that $\overline{\langle S_{n-1}^+,S_n\rangle}\subset\overline{\langle I_{k,n}^+\rangle}$. From \cite{BCF} one knows that $\overline{\langle S_{n-1}^+,S_n\rangle}=S_{n}^+$. The latter result relies heavily on the solution of Banica's conjecture for $n=5$.

The precise meaning of the aforementioned ideas is described in the course of the paper. \autoref{sec:preliminaries} serves mainly as preliminaries needed to settle the notation for compact quantum groups (\autoref{sec:cqg}), Hopf images (\autoref{sec:hopfimage}) and quantum permutation groups together with quantum increasing sequences (\autoref{sec:qpg-qis}). The proof of the main Theorem is outlined in \autoref{sec:main}.

\section{Preliminaries}\label{sec:preliminaries}
\subsection{Compact quantum groups}\label{sec:cqg} This part is devoted to specifying notation and some viewpoints in the theory of compact quantum groups. For details, we refer to \cite{SLW95b,Timm,NT13}.
\begin{df}\label{df:cqg}
A \emph{compact quantum group} is defined by its \emph{continuous functions algebra} $C(\GG)$ and a \emph{coproduct} $\Delta_\GG\colon C(\GG)\to C(\GG)\otimes C(\GG)$, where $C(\GG)$ is a unital $C^*$-algebra and $\Delta_\GG$ is a unital ${}^*$-homomorphism satisfying the Podleś conditions: 
\begin{align*}
\spn\left\{(\mathds{1}\otimes a)\Delta_\GG(a'):a,a'\in C(\GG)\right\}^\nc=C(\GG)\otimes C(\GG) \\
 \spn\left\{(a\otimes \mathds{1})\Delta_\GG(a'):a,a'\in C(\GG)\right\}^\nc=C(\GG)\otimes C(\GG) 
\end{align*}
The tensor product on the right hand side of $\Delta_{\GG}$ is the $C^*$-algebraic minimal tensor product.
\end{df}
\begin{df}\label{df:corep}
 An $n$-dimensional \emph{representation} of $\GG$ is a corepresentation of $C(\GG)$, i.e. an element $u\in\BC{n}\otimes C(\GG)$ such that \[\Delta_\GG(u_{ij})=\sum_{k=1}^nu_{ik}\otimes u_{kj},\]
 where $u_{ij}=(\langle e_i|\cdot|e_j\rangle\otimes\id)u$ and $(e_i)_{1\leq i\leq n}\subset \mathds{C}^n$ denote the standard orthonormal basis of $\mathds{C}^n$.
\end{df}
Developing the representation theory of $\GG$ one shows that there exists a dense Hopf-$^*$-algebra $\Pol(\GG)\subseteq C(\GG)$. It has numerous $C^*$-completions ($C(\GG)$ being one of them), but in particular it has the universal $C^*$-completion $\Cu{G}$. The difference is like between $C^*_{\max}(\mathds{F}_2)$ and $C^*_{\mathnormal{r}}(\mathds{F}_2)$ -- even though they are different as $C^*$-algebras, the group object under the carpet is the same. We will adopt this viewpoint when studying compact quantum groups and will not dig into the issues of different completions (for further discussion on this topic, see e.g. \cite{KS12}). We will just assume that the compact quantum group $\GG$ is studied by $\Cu{G}$, as one can always replace the original completion with the universal one.
\begin{df}\label{df:cmqg}
A \emph{compact matrix quantum group} is a pair $\left(\Cu{G},u\right)$ such that $u\in\BC{n}\otimes\Cu{G}$ is a representation of $\GG$ and $\{u_{ij}:1\leq i,j\leq n\}$ generates $\Pol(\GG)$ as ${}^*$-algebra. Such an element $u$ is called a \emph{fundamental corepresentation} of $\GG$.
\end{df}
\begin{df}\label{df:subgroup}
 Given two compact quantum groups $\HH,\GG$ we say that \emph{$\HH$ is embedded into $\GG$} (or shorter: $\HH\subset\GG$ is a closed quantum subgroup), if there exists a surjective unital ${}^*$-homomorphism $\pi\colon\Cu{G}\to\Cu{H}$ intertwining the respective coproducts: \[(\pi\otimes\pi)\circ\Delta_\GG=\Delta_\HH\circ\pi.\]
\end{df}
Note that the same compact quantum group $\HH$ can be embedded into $\GG$ in a not necessarily unique way, thus when we write $\HH\subset\GG$ we mean an inclusion specified by a particular $\pi$.


\subsection{Hopf image}\label{sec:hopfimage}
The Hopf image, studied in detail in the case of compact quantum groups in \cite{BB10, SS16} and in the case of locally compact quantum groups in \cite{JKS}, is concerned with the following situation. Let $\GG$ be a compact quantum group. Let $\mathnormal{B}$ be a unital $C^{\ast}$-algebra and let $\beta\colon\Cu{G}\to\mathnormal{B}$ be a unital ${}^*$-homomorphism. We think of it as the Gelfand dual of a map $\widehat{\beta}\colon\XX\to\GG$ from a quantum space into a quantum group and ask what is the closed quantum subgroup of $\GG$ generated by $\widehat{\beta}(\mathbb{X})\subset\GG$. We will abuse the notation and write $\Cu{X}$ instead of $\mathnormal{B}$.

\noindent\begin{minipage}{.58\textwidth}
\begin{flushleft}
\hspace*{6mm} Formally speaking, we consider the following category, which we denote by $\mathcal{C}_{\beta}$. Objects of $\mathcal{C}_{\beta}$ are triples $(\pi, \HH, \tilde{\beta})$ consisting of: a closed quantum subgroup $\HH$ of $\GG$ such that $\pi\colon\Cu{G}\to\Cu{H}$ is the associated ${}^*$-homomorphism intertwining the coproducts and $\tilde{\beta}\circ\pi=\beta$, i.e.~the map $\beta$ factors through the $C^*$-algebra $\Cu{H}$ of functions on the subgroup $\HH$ (where $\HH$ is embedded into $\GG$ via $\pi$) and $\tilde{\beta}\circ\pi=\beta$ is the factorisation. For two objects $\mathds{h}=(\pi, \HH, \tilde{\beta})$, $\mathds{k}=(\pi', \KK, \beta')\in \Ob(\mathcal{C}_{\beta})$, a morphism $\varphi\in \Mor_{\mathcal{C}_{\beta}}(\mathds{h},\mathds{k})$ is a unital ${}^*$-homomorphism $\varphi\colon\Cu{K}\to\Cu{H}$ (note that the direction of arrows in $\mathcal{C}_{\beta}$ matches the direction of maps on the level of quantum groups, which is opposite to the ones on the level of $C^*$-algebras of functions), which intertwines the respective coproducts and such that the diagram on the right commutes. 
\end{flushleft}
\end{minipage}
\begin{minipage}{.4\textwidth}
\begin{flushright}
\begin{tikzpicture}
  [bend angle=36,scale=2,auto,
pre/.style={<-,shorten <=1pt,semithick},
post/.style={->,shorten >=1pt,semithick}]
\node (G) at (-1,1) {$\Cu{G}$};
\node (B) at (1,1) {$\Cu{X}$}
edge [pre] node[auto,swap] {$\beta$} (G);
\node (H) at (0,0) {$\Cu{H}$}
edge [pre] node[auto,swap] {$\pi$} (G)
edge [post] node[auto] {$\tilde{\beta}$} (B);
\node (K) at (-1,-1) {$\Cu{K}$}
edge [pre] node[auto,swap] {$\pi'$} (G)
edge [post, bend right] node[auto,swap] {$\tilde{\beta'}$} (B)
edge [post] node[auto] {$\varphi$} (H);
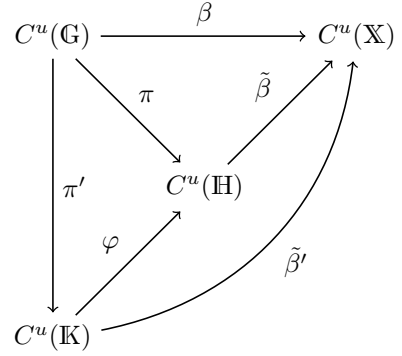
 \end{tikzpicture}\vskip1em\centering\hypertarget{diagram}{Diagram 1: Morphisms in $\mathcal{C}_{\beta}$}
 \end{flushright}
\end{minipage}

The object we are interested in is the initial object of the category $\mathcal{C}_{\beta}$. This object is, classically, the closed subgroup generated by $\XX$ and thus we will denote this initial object by $\overline{\langle\XX\rangle}$. We will use the universal property as follows. Let $\beta\colon\Cu{G}\to\Cu{X}, b\colon\Cu{H}\to\Cu{Y}$ be two morphisms and let
\[\Cu{G}\xrightarrow{\pi_{\overline{\langle\XX\rangle}}}C^u\left(\overline{\langle\XX\rangle}\right)\xrightarrow{\tilde{\beta}}\Cu{X}\quad\text{and}\quad\Cu{H}\xrightarrow{\pi_{\overline{\langle\YY\rangle}}}C^u\left(\overline{\langle\YY\rangle}\right)\xrightarrow{\tilde{b}}\Cu{Y}\]

\noindent\begin{minipage}{0.34\textwidth}
be the Hopf image factorisations of the morphisms $\beta, b$, respectively. Furthermore, assume that $\HH\subset\GG$ and $\YY\subset\XX$, i.e. there exists maps $\pi_{\HH}$ and $\pi_{\YY}$ such that the diagram on the right commutes. 

\indent The universal property of the Hopf image $\overline{\langle\XX\rangle}$ yields the existence of the map $q$ in the diagram on the right. We summarise this observation as:
\end{minipage}
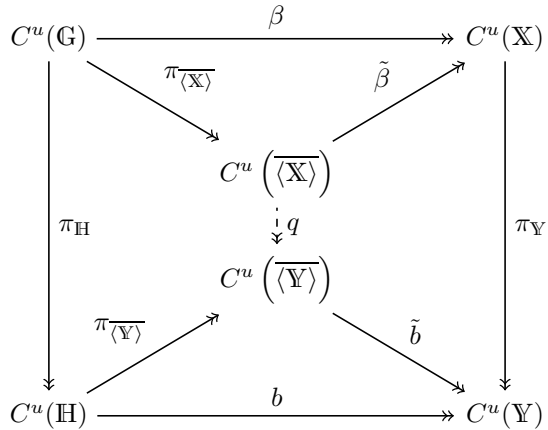
\begin{minipage}{0.65\textwidth}
\begin{center} \begin{tikzpicture}
  [bend angle=36,scale=2,auto,
pre/.style={<<-,shorten <=1pt,semithick},
post/.style={->>,shorten >=1pt,semithick}]
\node (G) at (-1.5,1.5) {$\Cu{G}$};
\node (X) at (1.5,1.5) {$\Cu{X}$}
edge [pre] node[auto,swap] {$\beta$} (G);
\node (<X>) at (0,0.6) {$C^u\left(\overline{\langle\XX\rangle}\right)$}
edge [pre] node[auto,swap] {$\pi_{\overline{\langle\XX\rangle}}$} (G)
edge [post] node[auto] {$\tilde{\beta}$} (X);
\node (<Y>) at (0,-0.1) {$C^u\left(\overline{\langle\YY\rangle}\right)$}
edge [dashed,pre] node[auto,swap] {$q$} (<X>);
\node (H) at (-1.5,-1) {$\Cu{H}$}
edge [pre] node[auto,swap] {$\pi_{\HH}$} (G)
edge [post] node[auto] {$\pi_{\overline{\langle\YY\rangle}}$} (<Y>);
\node (Y) at (1.5,-1) {$\Cu{Y}$}
edge [pre] node[auto,swap] {$b$} (H)
edge [pre] node[auto,swap] {$\tilde{b}$} (<Y>)
edge [pre] node[auto,swap] {$\pi_{\YY}$} (X);
\end{tikzpicture}\end{center}
\end{minipage}

\begin{fact}\label{fact:universal}
 If $\YY\subset\XX$, $\YY\subset\HH$, $\XX\subset\GG$ and $\HH\subset\GG$ is a closed quantum subgroup, then $\overline{\langle\YY\rangle}\subset\overline{\langle\XX\rangle}$.
\end{fact}

Note that the role of $\HH$ could be equally well played by $\overline{\langle\YY\rangle}$ in the above diagram (leading to minor simplifications). The reason we present it in this way is to emphasise that the outer arrows of morphisms of the diagram are the ones one can work with directly (by means of formulas), and the inner arrows of morphisms are the ones whose existence follows from universal property and the formulas need not be grasped.

\subsection{The Quantum Permutation Group \texorpdfstring{$S_n^+$}{} and Quantum Increasing Sequences \texorpdfstring{$I_{k,n}^+$}{}}\label{sec:qpg-qis}
Quantum permutation groups $S_n^+$ were introduced in \cite{Wang98} (cf. \cite[Section 3]{SS16}). Consider the universal $C^*$-algebra generated by $n^2$-elements $u_{ij}$, $1\leq i,j\leq n$ subject to the following relations:
\begin{enumerate}
 \item the generators $u_{ij}$ are all projections.
 \item $\sum_{i=1}^nu_{ij}=\mathds{1}=\sum_{j=1}^nu_{ij}$.
\end{enumerate}
This $C^*$-algebra will be denoted $\CuS{n}$. The matrix $u=[u_{ij}]_{1\leq i,j\leq n}$ is a fundamental corepresentation of $\CuS{n}$, this gives all the quantum group-theoretic data. Moreover, $S_n^+=S_n$ for $n\leq 3$ and $S_n^+\supsetneq S_n$ for $n\geq4$.

The algebra of continuous functions on the set of quantum increasing sequences was defined by Curran in \cite[Definition 2.1]{Cur11}. Let $k\leq n\in\mathds{N}$ and let $C^u(I^+_{k,n})$ be the universal $C^{\ast}$-algebra generated by $p_{ij}$, $1\leq i\leq n$, $1\leq j\leq k$ subject to the following relations:
\begin{enumerate}
 \item the generators $p_{ij}$ are all projections.
 \item each column of the rectangular matrix $p=[p_{ij}]_{1\leq i\leq n,1\leq j\leq k}$ forms a partition of unity: \(\sum_{i=1}^n p_{ij}=\mathds{1}\) for each \(1\leq j\leq k\).
 \item increasing sequence condition: \(p_{ij}p_{i'j'}=0\) whenever $j<j'$ and $i\geq i'$.
\end{enumerate}
This definition is obtained by the liberation philosophy (see \cite{BS09}): if one denotes by $I_{k,n}$ the set of increasing sequences of length $k$ and values in $\{1,\ldots,n\}$, then it is possible to write a matrix representation: to an increasing sequence $\underline{i}=(i_1<\ldots<i_k)$ one associates its matrix representation $A(\underline{i})\in M_{n\times k}(\{0,1\})$ as follows: $A(\underline{i})_{i_l,l}=1$ and all other entries are set to be $0$. One can check that the space of continuous functions on these matrices $C(\{A(\underline{i}):\underline{i}\in I_{k,n}\})$ is generated by the coordinate functions $x_{i,j}$ subject to the relations introduced above \textbf{and} the commutation relation (cf. the discussion after \cite[Remark 2.2]{Cur11}).

Curran defined also a ${}^{\ast}$-homomorphism $\beta_{k,n}\colon C(S_n^+)\to C(I^+_{k,n})$ (\cite[Proposition 2.5]{Cur11}) by:
\begin{itemize}
 \item $u_{ij}\mapsto p_{ij}$ for $1\leq i \leq n$, $1\leq j\leq k$,
 \item $u_{i\,k+m}\mapsto 0$ for $1\leq m\leq n-k$ and $i<m$ or $i>m+k$,
 \item for $1\leq m\leq n-k$ and $0\leq p\leq k$, \[ u_{m+p\,k+m}\mapsto \sum_{i=0}^{m+p-1} p_{ip}-p_{i+1\,p+1},\]
 where we set $p_{00}=\mathds{1}$, $p_{i0}=p_{0i}=p_{i\,k+1}=0$ for $i\geq1$.
\end{itemize}
This ${}^{\ast}$-homomorphism is well defined thanks to the additional relations obtained in \cite[Proposition 2.4]{Cur11} and the universal property of $\CuS{n}$. The map $\beta_{k,n}$ are defined in such a way that when applied to the commutative $C^*$-algebras $C(S_n)\to C(I_{k,n})$ (which satisfy the same relations plus commutativity), it is precisely the ``completing an increasing sequence to a permutation'' map. More precisely, one draws the diagram of an increasing sequence $\underline{i}=(i_1<\ldots<i_k)$ in the following way: drawing $k$ dots in one row and additional $n$ dots in the row below, one connects $l$-th dot in the upper row to the $i_l$-th dot in the lower row. Then one draws additional $n-k$ dots in the upper row next to previously drawn $k$ dots and connects them as follows: $(k+j)$-th dot is connected to the $j$-th leftmost non-connected dot in the bottom row. Finally, one obtains the diagram of a permutation on $n$ letters, which is then called $\beta_{k,n}(\underline{i})$ (for the version of $\beta_{k,n}$ as a map between appropriate commutative $C^*$-algebras), see the example below.
  
 \begin{figure}[h]
 \begin{tikzpicture}[scale=.8]
\foreach \i in {2,...,10} {
\path (\i,0) coordinate (P\i);}
\foreach \i in {2,...,10} {
\path (\i,1) coordinate (Q\i);}

\foreach \i in {7,...,10} \draw (Q\i) circle (3pt);

\draw[thick,dashed] (Q7) -- (P2);
\draw[thick,dashed] (Q8) -- (P5);
\draw[thick,dashed] (Q9) -- (P8);
\draw[thick,dashed] (Q10) -- (P10);

\draw[thick] (Q2) -- (P3);
\draw[thick] (Q3) -- (P4);
\draw[thick] (Q4) -- (P6);
\draw[thick] (Q5) -- (P7);
\draw[thick] (Q6) -- (P9);

\foreach \i in {2,...,10} \fill (P\i) circle (3pt);
\foreach \i in {2,...,6}  \fill (Q\i) circle (3pt);
\node[text width=0.9\textwidth,anchor=west] at (-2.5,-1) {\footnotesize{\textsc{FIGURE}. The sequence $(2\!<\!3\!<\!5\!<\!6\!<\!8)\in I_{5,9}$, drawn with full circles and segments, is completed, with the aid of empty circles and dashed segments, to the permutation $(1,2,3,5,8,7,4,6)$.}};
 \end{tikzpicture}\end{figure}

\begin{lemma}\label{lem:IS_gen}
Consider $I_{k,n}\subseteq S_n$ via the above map. Then:
\begin{enumerate}
\item $\langle I_{k,n}\rangle=S_n$ for all $n$ and all $k\neq0,n$,
\item $S_n\subset \overline{\langle I_{k,n}^+\rangle}$ for $k,n\neq0,n$. 
\end{enumerate}
\end{lemma}
\begin{proof}
 The first item is routine: by considering increasing sequences such that $a_l=l$ for $l<k$, and $a_k=k+p$ for $p\in\{1,n-k\}$, we obtain the transposition $(k,k+1)$ and a the cycle $(k,k+1,\ldots,n)$, which generate $S_{n-k+1}$ operating on last indices. Then by considering the sequence $a_l=l+1$ for all $l\leq k$, we obtain the cycle $(1,2,\ldots,k+1)$, together with the transposition $(k,k+1)$ they generate $S_{k+1}$ operating on first $k+1$ indices. Consequently, we have generated a subgroup of $S_n$ containing $S_{k+1}$ operating on $\{1,\ldots,k+1\}$ and $S_{n-k+1}$ operating on $\{k,k+1,\ldots,n\}$, the two generate whole $S_n$.

 The second item follows easily from \autoref{fact:universal}.
\end{proof}
\section{Quantum Increasing Sequences generate Quantum Permutation Groups}\label{sec:main}
\begin{lemma}\label{lem:ext_range}
 Fix $n\in\mathds{N}$ and $1\leq k \leq n-1$. Let $(p_{ij})_{1\leq i\leq n, 1\leq j \leq k}$, $(\tilde{p}_{ij})_{1\leq i\leq n-1, 1\leq j \leq k}$ and $(\dot{p}_{ij})_{1\leq i\leq n-1, 1\leq j \leq k-1}$ denote the standard generators of $C^u(I_{k,n}^+)$, $C^u(I_{k,n-1}^+)$ and $C^u(I_{k-1,n-1}^+)$, respectively. The maps $\tilde{\eta}_{k,n}\colon C^u(I_{k,n}^+)\to C^u(I_{k,n-1}^+)$ and $\dot{\eta}_{k,n}\colon C^u(I_{k,n}^+)\to C^u(I_{k-1,n-1}^+)$ determined by:

 \noindent
\begin{minipage}{.54\textwidth}
\begin{flushright}
 \[\dot{\eta}_{k,n}(p_{ij})=\left\{\begin{array}{rl}
           \mathds{1}&\quad i=j=1\\
           0&\quad i=1\text{ or }j=1\text{ and }i\neq j\\
           \dot{p}_{i-1\,j-1} &\quad 1<i\leq n, 1< j\leq k
          \end{array}\right.\]
\end{flushright}
\end{minipage}
\begin{minipage}{.46\textwidth}
\begin{flushleft}
\[\tilde{\eta}_{k,n}(p_{ij})=\left\{\begin{array}{rl}
           \tilde{p}_{ij}&\quad i<n\\
           0&\quad i=n
          \end{array}\right.\]
\end{flushleft}
\end{minipage}
are well-defined surjective ${}^*$-homomorphisms.
\end{lemma}
\begin{proof}
In both cases the existence of the ${}^*$-homomorphisms is a consequence of the universal property of $C^*(I_{k,n}^+)$, as long as relations satisfied by $p_{ij}$ are satisfied by their images.

Clearly, $\tilde{\eta}_{k,n}(p_{ij})$ are all projections, which add up to $\mathds{1}$ in each column. The increasing sequence condition amounts to verifying if \[\forall_{j<j'}\forall_{i\geq i'}\quad \tilde{\eta}_{k,n}(p_{ij})\tilde{\eta}_{k,n}(p_{i'j'})=0.\]
 If $i=n$, this is obvious, as $\tilde{\eta}_{k,n}(p_{ij})=0$ from the definition. Otherwise, $n>i\geq i'$ and consequently \[\tilde{\eta}_{k,n}(p_{ij})\tilde{\eta}_{k,n}(p_{i'j'})=\tilde{p}_{ij}\tilde{p}_{i'j'}=0\] thanks to increasing sequence condition satisfied by $\tilde{p}_{ij}$'s in $C^u(I_{k,n-1}^+)$.

Likewise, $\dot{\eta}_{k,n}(p_{ij})$ are all projections, which add up to $\mathds{1}$ in each column, and $\dot{\eta}_{k,n}(p_{ij})\dot{\eta}_{k,n}(p_{i'j'})$ satisfy the increasing sequence condition for $n\geq i\geq i'>1$ and $1<j<j'$ due to relations in $C^u(I_{k-1,n-1}^+)$. If $i'=1$, then $1\leq j< j'$ and consequently $\dot{\eta}_{k,n}(p_{i'j'})=0$ and the increasing sequence condition follows. Similarly, if $j=1$ and $i\geq i'>1$, then $\dot{\eta}_{k,n}(p_{ij})=0$ and the increasing sequence conditions follows.\end{proof}
\begin{lemma}\label{lem:ext_range2}
 Fix $n\in\mathds{N}$ and $1\leq k \leq n-1$ and let $\tilde{\eta}_{k,n}, \dot{\eta}_{k,n}$ be the maps from \autoref{lem:ext_range}. Let $(u_{ij})_{1\leq i,j\leq n}$ and $(u_{ij}')_{1\leq i,j\leq n-1}$ be the standard generators of $\CuS{n}$ and $\CuS{n-1}$, respectively, and let $q_n, \bar{q}_n\colon\CuS{n}\to\CuS{n-1}$ be the maps 
 
 \noindent
\begin{minipage}{.5\textwidth}
\begin{flushleft}
\[q_n(u_{ij})=\left\{\begin{array}{rl}
           u_{ij}'&\quad i,j<n\\
           \mathds{1}&\quad i=j=n\\
           0&\quad i=n, j<n\text{ or }i<n,j=n
          \end{array}\right.\]
(1) The diagram below is commutative:
  \begin{tikzpicture}
  [bend angle=36,scale=2,auto,
pre/.style={<-,shorten <=1pt,semithick},
post/.style={->,shorten >=1pt,semithick}]
\node (Sn) at (-1,1) {$\CuS{n}$};
\node (Ikn) at (1,1) {$C^u(I_{k,n}^+)$}
edge [pre] node[auto,swap] {$\beta_{k,n}$} (Sn);
\node (Ikn-1) at (1,-0.5) {$C^u(I_{k,n-1}^+)$}
edge [pre] node[auto,swap] {$\tilde{\eta}_{k,n}$} (Ikn);
\node (Sn-1) at (-1,-0.5) {$\CuS{n-1}$}
edge [pre] node[auto,swap] {$q_n$} (Sn)
edge [post] node[auto,swap] {$\beta_{k,n-1}$} (Ikn-1);
 \end{tikzpicture} 
\end{flushleft}
\end{minipage}
\begin{minipage}{.5\textwidth}
\begin{flushright}
\[\bar{q}_n(u_{ij})=\left\{\begin{array}{rl}
           u_{i-1\,j-1}'&\quad i,j>1\\
           \mathds{1}&\quad i=j=1 \\
           0&\quad i=1, j>1\text{ or }i>1,j=1
          \end{array}\right.\]
 (2) The diagram below is commutative:
  \begin{tikzpicture}
  [bend angle=36,scale=2,auto,
pre/.style={<-,shorten <=1pt,semithick},
post/.style={->,shorten >=1pt,semithick}]
\node (Sn) at (-1,1) {$\CuS{n}$};
\node (Ikn) at (1,1) {$C^u(I_{k,n}^+)$}
edge [pre] node[auto,swap] {$\beta_{k,n}$} (Sn);
\node (Ik-1n-1) at (1,-0.5) {$C^u(I_{k-1,n-1}^+)$}
edge [pre] node[auto,swap] {$\dot{\eta}_{k,n}$} (Ikn);
\node (Sn-1) at (-1,-0.5) {$\CuS{n-1}$}
edge [pre] node[auto,swap] {$\bar{q}_n$} (Sn)
edge [post] node[auto,swap] {$\beta_{k-1,n-1}$} (Ik-1n-1);
 \end{tikzpicture}
\end{flushright}
\end{minipage}
\end{lemma}
\begin{proof}
As in \autoref{lem:ext_range}, let $(p_{ij})_{1\leq i\leq n, 1\leq j \leq k}$, $(\tilde{p}_{ij})_{1\leq i\leq n-1, 1\leq j \leq k}$ and $(\dot{p}_{ij})_{1\leq i\leq n-1, 1\leq j \leq k-1}$ denote the standard generators of $C^u(I_{k,n}^+)$, $C^u(I_{k,n-1}^+)$ and $C^u(I_{k-1,n-1}^+)$, respectively. 

\noindent (1) One needs to check if $\tilde{\eta}_{k,n}(\beta_{k,n}(u_{ij})) = \beta_{k,n-1}(q_n(u_{ij}))$ for all $1\leq i,j\leq n$. As the Curran's map is defined by some special cases, we will analyse if this equality holds in each of the cases separately.
\begin{enumerate}[(a)]
 \item $1\leq i \leq n, 1\leq j \leq k$. 
 \[\tilde{\eta}_{k,n}(\beta_{k,n}(u_{ij})) = \tilde{\eta}_{k,n}(p_{ij}) = \left\{\begin{array}{rl} \tilde{p}_{ij}&\quad i<n \\ 0&\quad i=n \end{array}\right.\]
 On the other hand, \[\beta_{k,n-1}(q_n(u_{ij})) = \left\{\begin{array}{rll} \beta_{k,n-1}(u'_{ij}) =& \tilde{p}_{ij} &\quad i<n \\ \beta_{k,n-1}(0) =& 0 &\quad i=n \end{array}\right.\]
 \item $1\leq m \leq n-k, i<m\text{ or }i>m+k$. 
 \[\tilde{\eta}_{k,n}(\beta_{k,n}(u_{i\,k+m})) = \tilde{\eta}_{k,n}(0) = 0\]
 On the other hand, \[\beta_{k,n-1}(q_n(u_{i\,k+m})) = \left\{\begin{array}{rll} \beta_{k,n-1}(u'_{i\,k+m}) =& 0 &\quad \LARGE\substack{1\leq m \leq n-1 -k\\ i<m \text{ or } n-1\geq i>m+k} \\ \beta_{k,n-1}(0) =& 0 &\quad \LARGE\substack{m\leq n-1-k, i=n \\\text{ or }m=n-k} \end{array}\right.\]  
 \item $1\leq m \leq n-k, 0\leq p\leq k$. We start with the right hand side:
 \[
 \beta_{k,n-1}(q_n(u_{m+p\,k+m})) = \left\{\begin{array}{>{\displaystyle}r>{\displaystyle}ll}\beta_{k,n-1}(u'_{m+p\,k+m}) =& \sum_{i=0}^{m+p-1}(\tilde{p}_{ip}-\tilde{p}_{i+1\,p+1}) &\quad \LARGE\substack{m\leq n-1-k\\ 0\leq p\leq k} \\ \beta_{k,n-1}(\mathds{1}) =& \mathds{1} &\quad m=n-k, p=k \\ \beta_{k-1,n}(0) =& 0 &\quad m=n-k, p<k \end{array}\right.
 \]
 Whereas the left hand side:
 \begin{align*} \tilde{\eta}_{k,n}(\beta_{k,n}(u_{m+p\,k+m}))&=\sum_{i=0}^{m+p-1}\tilde{\eta}_{k,n}(p_{ip})-\tilde{\eta}_{k,n}(p_{i+1\,p+1}) \\
 &= \left\{\begin{array}{>{\displaystyle}r*2{>{\displaystyle}l}l}
	    \sum_{i=0}^{m+p-1}(\tilde{p}_{ip}-\tilde{p}_{i+1\,p+1}) & {} & {} & \quad \LARGE\substack{m\leq n-1-k\\ 0\leq p\leq k} \\
	    \sum_{i=0}^{n-2}\left(\tilde{p}_{ik}-\tilde{p}_{i+1\,k+1}\right) +\tilde{p}_{n-1\,k} &=\sum_{i=0}^{n-1}\tilde{p}_{ik}&=\mathds{1} &\quad\LARGE\substack{ m=n-k\\ p=k}\\
            \sum_{i=0}^{n-1-k}\tilde{p}_{i0}-\tilde{p}_{i+1\,1} &=\text{(i)}&=0 &\quad\LARGE\substack{ m=n-k\\ p=0}\\
            \sum_{i=0}^{n-k+p-1}\tilde{p}_{ip}-\tilde{p}_{i+1\,p+1}&=\text{(ii)} &=0 & \quad \LARGE\substack{m=n-k\\ 0<p<k}
           \end{array}\right. 
 \end{align*}
where in the case $m=n-k, p=k$ we used that $p_{i\,k+1}=0$. Further:
\begin{enumerate}[(i)]
\item $m=n-k, p=0$. By \cite[Lemma 2.4]{Cur11} $\tilde{p}_{i1}=0$ for $i>n-k$, thus \(\displaystyle\sum_{i=0}^{n-k}\tilde{p}_{i1}=\sum_{i=0}^{n-1}\tilde{p}_{i1}=\mathds{1}\). Consequently
\[\sum_{i=0}^{n-1-k}\tilde{p}_{i0}-\tilde{p}_{i+1\,1}=\mathds{1}-\sum_{i=1}^{n-k}\tilde{p}_{i1}=\mathds{1}-\mathds{1}=0\]
where we used that $p_{00}=\mathds{1}$ and $p_{i0}=0$.
\item $m=n-k, 0<p<k$. 
\[\sum_{i=0}^{n-k+p-1}\tilde{p}_{ip}-\tilde{p}_{i+1\,p+1}=\sum_{i=1}^{n-1-k+p}\tilde{p}_{ip}-\sum_{i=1}^{n-k+p}\tilde{p}_{i\,p+1}=\sum_{i=1}^{n-1}\tilde{p}_{ip}-\sum_{i=1}^{n-1}\tilde{p}_{i\,p+1}=\mathds{1}-\mathds{1}=0\]
as $\tilde{p}_{ip}=0$ for $i>n-k+p-1$ and $\tilde{p}_{i\,p+1}=0$ for $i> n-k+p$ by \cite[Lemma 2.4]{Cur11}. 
\end{enumerate}
\end{enumerate}

\noindent (2) One needs to check if $\dot{\eta}_{k,n}(\beta_{k,n}(u_{ij})) = \beta_{k-1,n-1}(\bar{q}_n(u_{ij}))$ for all $1\leq i,j\leq n$. 
\begin{enumerate}[(a)]
 \item $1\leq i \leq n, 1\leq j \leq k$. 
 \[\dot{\eta}_{k,n}(\beta_{k,n}(u_{ij})) = \dot{\eta}_{k,n}(p_{ij}) = \dot{\eta}_{k,n}(p_{ij})=\left\{\begin{array}{rl}
           \mathds{1}&\quad i=j=1\\
           0&\quad i=1\text{ or }j=1\text{ and }i\neq j\\
           \dot{p}_{i-1\,j-1} &\quad 1<i\leq n, 1< j\leq k
          \end{array}\right.\]
 On the other hand, \[\beta_{k-1,n-1}(\bar{q}_n(u_{ij})) = 
 \left\{\begin{array}{rll} 
 \beta_{k-1,n-1}(\mathds{1}) =& \mathds{1} &\quad i=j=1 \\ 
 \beta_{k-1,n-1}(0) =& 0 &\quad i=1\text{ or }j=1\text{ and }i\neq j\\
 \beta_{k-1,n-1}(u_{i-1\,j-1}') =& \dot{p}_{i-1\,j-1} &\quad 1<i\leq n, 1< j\leq k
 \end{array}\right.\]
 \item $1\leq m \leq n-k, i<m\text{ or }i>m+k$. 
 \[\dot{\eta}_{k,n}(\beta_{k,n}(u_{i\,k+m})) = \dot{\eta}_{k,n}(0) = 0\]
 On the other hand, 
 \[\beta_{k-1,n-1}(\bar{q}_n(u_{i\,k+m})) = 
 \left\{\begin{array}{rll} 
 \beta_{k-1,n-1}(0) =& 0 &\quad i=1\\
 \beta_{k-1,n-1}(u_{i-1\,k+m-1}') =& 0 &\quad\LARGE\substack{1\leq m \leq n-k\\ 2\leq i\leq m\text{ or }i>m+k}
 \end{array}\right.
 \]  
 Note that the bottom part of the above displayed equation covers also the case $i=m$, which will be necessary in the consecutive part. More precisely, for $2\leq m \leq n-k$, we have
 \[\beta_{k-1,n-1}(\bar{q}_n(u_{m\,k+m})) = \beta_{k-1,n-1}(u_{m-1\,k+m-1}') = 0.\]
 \item $1\leq m \leq n-k, 0\leq p\leq k$. We have
 \begin{align*} {}&\dot{\eta}_{k,n}(\beta_{k,n}(u_{m+p\,k+m}))=\sum_{i=0}^{m+p-1}\dot{\eta}_{k,n}(p_{ip})-\dot{\eta}_{k,n}(p_{i+1\,p+1}) \\
 &= \left\{\begin{array}{>{\displaystyle}r*2{>{\displaystyle}l}l}
	    \sum_{i=0}^{m-1}\dot{\eta}_{k,n}(p_{i0})-\dot{\eta}_{k,n}(p_{i+1\,1}) &= \dot{\eta}_{k,n}(p_{00})-\dot{\eta}_{k,n}(p_{1\,1}) &=\mathds{1}-\mathds{1}=0 & \quad p=0 \\
	    \sum_{i=0}^{m+p-1}\dot{\eta}_{k,n}(p_{i1})-\dot{\eta}_{k,n}(p_{i+1\,2}) &= \mathds{1} - \sum_{i=0}^{m+p-2}\dot{p}_{i+1\,1} &= \sum_{i=0}^{m+p-2}\dot{p}_{i\,0}-\dot{p}_{i+1\,1}& \quad p=1 \\
            \sum_{i=0}^{m+p-1}\dot{\eta}_{k,n}(p_{ip})-\dot{\eta}_{k,n}(p_{i+1\,p+1})&=\sum_{i=0}^{m+p-2}\dot{p}_{i\,p-1}-\dot{p}_{i+1\,p} &{} & \quad p\geq2
           \end{array}\right. 
 \end{align*}
 The case $p=0$ of the right hand side was covered at the end of item (b), we focus on $1\leq p\leq k$. Then $m+p\geq2$ and we have:
 \[ \beta_{k-1,n-1}(\bar{q}_n(u_{m+p\,k+m})) = \beta_{k-1,n-1}(u_{m+(p-1)\, (k-1)+m}) = \sum_{i=0}^{m+p-2}\dot{p}_{i\,p-1}-\dot{p}_{i+1\,p}\]
\end{enumerate}
This concludes the whole proof.
\end{proof}
\begin{theorem}
 If $n\geq4$ and $2\leq k \leq n-2$, then $\overline{\langle I_{k,n}^+\rangle} = S_n^+$.
\end{theorem}
\begin{proof}
 We will proceed inductively. Case $n=4$, $k=2$ was solved in \cite{remark}, let us assume $n\geq5$ and assume $\overline{\langle I_{k,n-1}^+\rangle} = S_{n-1}^+$ for all $2\leq k\leq n-3$. 
 
 Firstly, $S_n\subseteq\overline{\langle I_{k,n}^+\rangle}$ from \autoref{lem:IS_gen}. Then
 \begin{enumerate}
  \item if $k\leq n-3$, we can use \autoref{fact:universal} with $\XX= I_{k,n}^+$, $\YY = I_{k,n-1}^+$, $\pi_{\HH}=q_n$, $\beta=\beta_{k,n}$, $\pi_\YY=\tilde{\eta}_{k,n}$ and $b=\beta_{k,n-1}$ thanks to \autoref{lem:ext_range2},
  \item if $k\geq3$, we can use \autoref{fact:universal} with $\XX= I_{k,n}^+$, $\YY = I_{k-1,n-1}^+$, $\pi_{\HH}=\bar{q}_n$, $\beta=\beta_{k,n}$, $\pi_\YY=\dot{\eta}_{k,n}$ and $b=\beta_{k-1,n-1}$ thanks to \autoref{lem:ext_range2}.
 \end{enumerate}
In both cases we conclude that $S_{n-1}^+\subset \overline{\langle I_{k,n}^+\rangle}$. Together with the first observation, this yields $\overline{\langle S_n,S_{n-1}^+\rangle}\subset \overline{\langle I_{k,n}^+\rangle}\subset S_n^+$, but from \cite[Theorem 3.3]{BCF} it follows that all inclusions are in fact equalities.
\end{proof}

\bibliography{biblio}
\bibliographystyle{siam}
\end{document}